\documentclass[12pt]{amsart}
\usepackage{palatino}
\usepackage{amscd,amssymb}
\usepackage{float}
\usepackage{graphicx}

\usepackage[cmtip,matrix,arrow]{xy}

\usepackage{graphicx}

\numberwithin{equation}{section}

\newtheorem {Theorem}                   {Theorem}

\newtheorem {RefTheorem}[equation]      {Theorem}

\newtheorem {Lemma}[equation]           {Lemma}
\newtheorem {Proposition}[equation]     {Proposition}

\newtheorem {corollary}[equation]       {Corollary}

\theoremstyle{definition}

 \newtheorem{zremark}[equation] {Remark}
\newenvironment{Remark} {\par\footnotesize\zremark}{~\\}{\endzremark}
\newcommand{\pr} {\smallskip\noindent{\bf Proof\,\,}}

\newcommand     {\comment}[1]   {}
\newcommand     {\mute}[2] {}
\newcommand     {\printname}[1] {}

\newcommand{\labell}[1] {\label{#1}\printname{#1}}

\def    \ind    {{\operatorname{ind}}}

\def    \ind   {\operatorname{ind}}

\def    \to     {\longrightarrow}

\def    \R      {{\mathbb R}}

\def    \Z      {{\mathbb Z}}

\def    \pr     {\operatorname{pr}}

\begin{document}

\title[Hilbert polynomials of Calabi Yau hypersurfaces]{ Hilbert polynomials of Calabi Yau hypersurfaces in toric varieties and lattice points in polytope boundaries}

\author{Jonathan Weitsman}
\thanks{ }
\address{Department of Mathematics, Northeastern University, Boston, MA 02115}
\email{j.weitsman@neu.edu}
\thanks{\today}

\begin{abstract}  
We show that the Hilbert polynomial of a Calabi-Yau hypersurface $Z$ in a smooth toric variety $M$ associated to a convex polytope $\Delta$ is given by a lattice point count in the polytope boundary $\partial \Delta,$ just as the Hilbert polynomial of $M$ is known to be given by a lattice point count  in the convex polytope $\Delta.$ Our main tool is a computation of the Euler class in $K$-theory of the normal line bundle to the hypersurface $Z,$ in terms of the Euler classes of the divisors corresponding to the facets of the moment polytope.  We observe a remarkable parallel between our expression for the Euler class and the inclusion-exclusion principle in combinatorics.  To obtain our result we combine these facts with the known relation between lattice point counts in the facets of $\Delta$ and the Hilbert polynomials of the smooth toric varieties corresponding to these facets.
\end{abstract}

\maketitle

\section{Introduction} 

In this paper we show how the Hilbert polynomial of a Calabi-Yau hypersurface in a toric variety is associated to a lattice point count in a polytope boundary, in analogy with a similar result for toric varieties, giving the Hilbert polynomial of a toric variety in terms of lattice points in the associated convex polytope. We begin by recalling that association.

\subsection{Lattice points and toric varieties}

Let $(M,\omega)$ be a smooth, compact, connected Kahler toric variety of real dimension $2m$  equipped with a holomorphic Hermitian line bundle $L$ with Chern connection $\nabla$ of curvature $\omega.$ The manifold $M$ has an effective Hamiltonian action of a compact torus $T^m$, with moment map
\[
\mu: M \to \mathfrak{t}^* \cong \mathbb{R}^m
\]
of image given by a lattice polytope $\Delta \subset \mathbb{R}^m$. 

Since $M$ is smooth, the polytope $\Delta$ is a Delzant polytope; that is, the a lattice polytope where the primitive lattice edge vectors at each vertex form a basis for $\Z^m.$ Delzant's Theorem (\cite{delzant}) shows that any such polytope arises from a smooth toric variety in this way. 

Let $\text{ind } \bar{\partial}_{L^k}$ denote the index of the Dolbeault operator $\bar{\partial}_{L^k}$ \cite{shan} on sections on $L^k,$ for $k \in \mathbb{Z}_+.$ Geometric quantization and the principle of ``Quantization commutes with reduction'' [GS] then give  the Hilbert polynomial of $M$  in terms of lattice points in the polytope $\Delta:$

\begin{RefTheorem}[See e.g. \cite{Fulton, Guillemin}]\labell{tvp}
The Hilbert polynomial of $M$ is given by
\[
\dim H^0(M, L^k) = \text{ind } \bar{\partial}_{L^k} = \#(k\Delta \cap \mathbb{Z}^m)
\]
\end{RefTheorem}

On the other hand, the Riemann Roch theorem gives a way of computing $\text{ind } \bar{\partial}_{L^k}.$  To do so, we first describe the polytope $\Delta$  as an intersection of half spaces.  For $\lambda_i  \in \mathbb{R}_{\geq 0},$ $i=1,\dots,d,$ let
\[
  \Delta(\lambda_1 , \ldots, \lambda_d ) = \bigcap_{i=1}^d H_i(\lambda_i)
\]
be the intersection of half-spaces $H_i(\lambda_i)$ given by
\[
H_i(\lambda_i) = \{x \in \mathbb{R}^m : x \cdot n_i \leq \lambda_i\}, \quad i=1,\ldots,d
\]
where $n_i \in \mathbb{R}^m$ are primitive lattice vectors normal to the facets $F_i$ of $\Delta.$ Then for some nonnegative real numbers $\lambda_i^0,$ $ i = 1, \dots, d,$ 
we have

\[
\Delta = \Delta(\lambda_1^0 , \dots, \lambda_d^0 ).\] 

In view of Theorem \ref{tvp}, the Riemann-Roch formula then gives  Khovanskii's formula for the number of lattice points in a Delzant polytope (see Section \ref{section2} for a sketch of the proof):

\begin{RefTheorem}[Khovanskii \cite{k,kk,kp}] \labell{hlp} The number of lattice points in the polytope $\Delta$ associated to $M$ is given by
\begin{equation}\labell{khoveq}
 \#(\Delta \cap \mathbb{Z}^m)= \prod_{i=1}^d \text{Td}\left(\frac{\partial}{\partial \lambda_i}\right) \text{vol}(\Delta(\lambda_1, \dots,\lambda_d))\Big|_{ \lambda_i=\lambda_i^0}
\end{equation}

\noindent where the infinite order differential operators $\text{Td}(\frac{\partial}{\partial \lambda_i})$ are given by
\[
\text{Td}\left(\frac{\partial}{\partial \lambda_i}\right) = \sum_{j=0}^\infty \frac{b_j}{j!} \left(\frac{\partial}{\partial \lambda_i}\right)^j
\] and the coefficients $b_j$ are the Bernoulli numbers, given by the power series expansion at $0$ of the function

$$\text{Td}(x) = \frac{x}{1-e^{-x}}.$$
\end{RefTheorem}
\begin{Remark}\labell{infrem} Note that the volume $\text{vol}(\Delta(\lambda_1,\dots,\lambda_d))$ is a polynomial in the $\lambda_i,$ so there is no difficulty in applying the infinite order differential operators appearing in (\ref{khoveq}) to this volume function.\end{Remark}
\begin{Remark}  The generalization of Theorem \ref{hlp} to polytopes which are not Delzant is related to the geometry of singular toric varieties.  Some references are \cite{cs,bv, dr, guillemin,ksw}\end{Remark}

\subsection{Calabi-Yau hypersurfaces in toric varieties}

For each facet $F_i \in \Delta$, let $D_i = \mu^{-1}(F_i)$ be the corresponding complex codimension-one subvariety of $M$. Each subvariety $D_i$ corresponds to a line bundle $L_{F_i} \to M$. Batyrev \cite{batyrev} showed that, under certain conditions,  the divisor class $[D_{F_1}] + \dots + [D_{F_d}]$ corresponding to the line bundle $L_{F_1} \otimes \cdots \otimes {L_{F_d}}$ contains a smooth connected representative $Z,$ giving a Calabi-Yau hypersurface in $M$. This hypersurface is equipped with the holomorphic Hermitian line bundle $L|_Z$ and its powers and with a Dolbeault operator $\bar{\partial}_{({L^k}|_Z)} .$  The divisor class $[D_{F_1}] + \dots + [D_{F_d}]$ also has the singular representative $Z_{\rm sing}= \cup_{i=1}^d D_i$ whose moment image is $\cup_{i=1}^d F_i = \partial\Delta$.
 
The main result of this paper is that, in analogy to Theorem \ref{tvp}, we have a lattice point formula for the Hilbert polynomial of the Calabi-Yau hypersurface $Z:$
 
\begin{Theorem}\labell{mainthm} Suppose the divisor class $[D_{F_1}] + \dots + [D_{F_d}]$ corresponding to the line bundle $ L_{F_1} \otimes \cdots \otimes {L_{F_d}}$  has a representative given by a smooth connected complex hypersurface $Z \subset M.$
Then the Hilbert polynomial of $Z$ is given by
\[
\text{ind } \bar{\partial}_{({L^k}|_Z)} = \#(k(\partial\Delta) \cap \mathbb{Z}^m)
\]
\end{Theorem}

As a corollary, we have a geometric proof of the following analog of Khovanskii's Theorem, already proved by combinatorial methods in \cite{qcy}:
 
\begin{Theorem}[See \cite{qcy}]\labell{bdp}  Suppose the divisor class $[D_{F_1}] + \dots + [D_{F_d}]$ corresponding to the line bundle $ L_{F_1} \otimes \cdots \otimes {L_{F_d}}$  has a representative given by a smooth connected complex hypersurface $Z \subset M.$
The number of lattice points in the boundary $\partial\Delta$ of the polytope $\Delta$ is given by
\begin{equation}\labell{bdpeq}
\#(\partial\Delta \cap \mathbb{Z}^m) = \left(\prod_{i=1}^d \hat{A}\left(\frac{\partial}{\partial \lambda_i}\right) \right)\frac{1}{\hat{A}}\left(\sum_{i=1}^d\frac{\partial}{\partial \lambda_i}\right) \text{vol}(\partial\Delta(\lambda_1,\dots,\lambda_d))\Big|_{\lambda_i=\lambda_i^0}
\end{equation}
where the differential operators $\hat{A}(\frac{\partial}{\partial \lambda_i})$ and $\frac{1}{\hat{A}}(\sum_{i=1}^d\frac{\partial}{\partial \lambda_i})$ are given by
\[
\hat{A}\left(\frac{\partial}{\partial \lambda_i}\right) = \sum_{j=0}^\infty  c_{2j}  \left(\frac{\partial}{\partial \lambda_i}\right)^{2j}
\]
\[
\frac{1}{\hat{A}}\left(\sum_{i=1}^d\frac{\partial}{\partial \lambda_i}\right) =  \sum_{j=0}^\infty \frac{1}{2^{2j}(2j+1)!} \left(\sum_{i=1}^d\frac{\partial}{\partial \lambda_i}\right)^{2j};
\]

\noindent here the coefficients ${c_{2j}}$ are given by the  power series expansion at zero of the function
\[
\hat{A}(x) = \frac{x/2}{\sinh(x/2)} =  \sum_{j=0}^\infty {c_{2j}}x^{2j}\]

\noindent and similarly
\[
\quad \frac{1}{\hat{A}(x)} = \frac{\sinh(x/2)}{x/2} =  \sum_{j=0}^\infty \frac{1}{2^{2j}(2j+1)!} x^{2j}.
\]
Also, $\text{vol}(\partial\Delta(\lambda_1,\dots,\lambda_d))$ is defined to be the sum of the Euclidean volumes of the facets, each divided by the length of the primitive lattice vector normal to it (See e.g. \cite{robins}, Lemma 5.18).\footnote{Note that as in Remark \ref{infrem}, the volume $\text{vol}(\partial\Delta(\lambda_1,\dots,\lambda_d))$ is a polynomial in the $\lambda_i,$ so there is no difficulty in applying the infinite order differential operators appearing in (\ref{bdpeq}) to this volume function.}
\end{Theorem}

Note that Theorem \ref{bdp} is an entirely combinatorial result.  In \cite{qcy} we proved Theorem \ref{bdp} by pure combinatorial methods, based on the results of \cite{a,aw}.\footnote{The combinatorial argument in \cite{qcy} does not require any condition about a smooth representative for $\otimes_{i=1}^d L_{F_i}.$} We noted in \cite{qcy}, however, that Theorem \ref{bdp} should morally be provable by geometric methods, if we could imagine that $Z_{\text{sing}} = \cup_{i=1}^d D_i$ was a smooth manifold. However, this is not remotely the case, even for $M = \mathbb{CP}^2$ (See Figure 1); while the smooth deformations $Z$ of $Z_{\text{sing}}$ are not Hamiltonian $T^m$-spaces, so that the methods of equivariant topology used in the proof of Theorem \ref{hlp} do not apply. In this paper we show how computations in $K$ theory allow us to overcome the problems arising from the singularities of $Z_{\text{sing}}$ and the absence of a useful torus action on $Z$, and obtain the Hilbert polynomial of $Z$ (Theorem \ref{mainthm}) and as a corollary a {\em geometric} proof of Theorem \ref{bdp}.

\begin{figure} 
       \includegraphics[width=.5\textwidth]{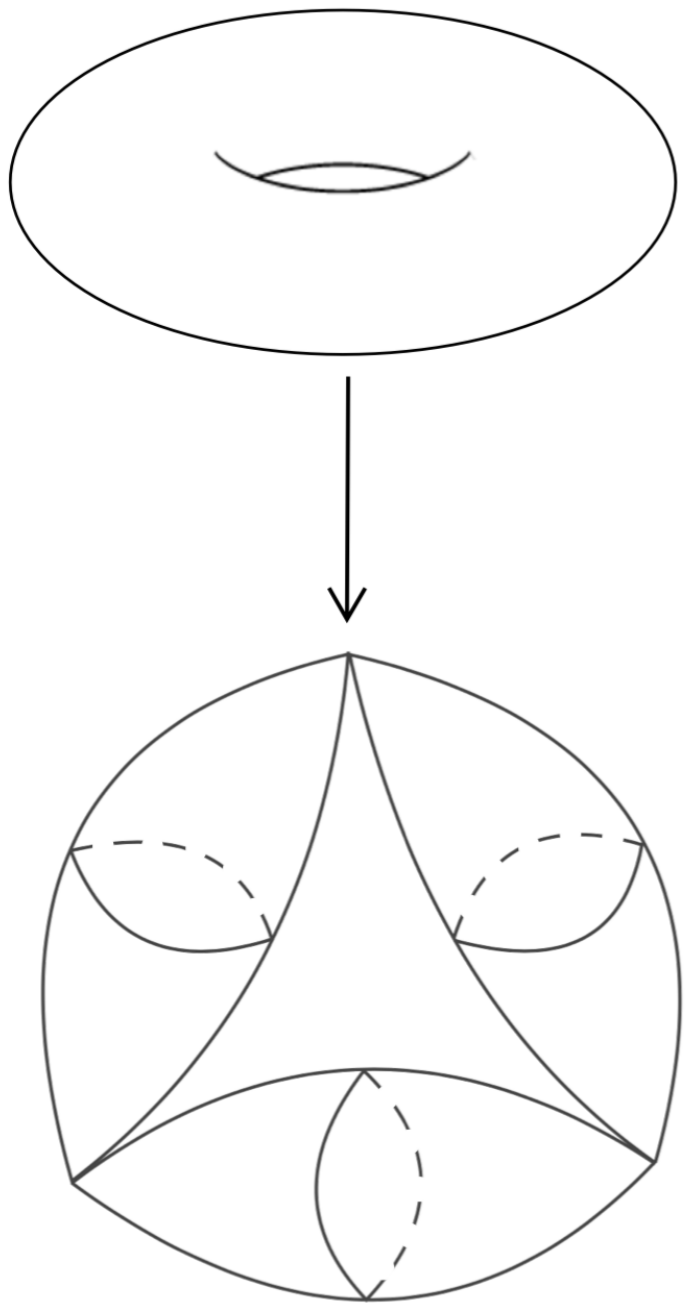}
\caption{The singular Calabi Yau $Z_{\rm sing} \subset {\mathbb CP}^2$ as a degeneration of a torus}
   \end{figure}

\begin{Remark}  From the point of view of geometric quantization, Theorem \ref{mainthm} shows that the geometrically defined Hilbert polynomial of $Z$ agrees with the formal geometric quantization (see \cite{weitsman, paradan})  of $Z$ used in \cite{qcy}.  \end{Remark}

As a Corollary of Theorem \ref{mainthm}, the combinatorial computations in Section 3.3 of \cite{qcy} of lattice point counts in boundaries of low dimensional simplices give the following computations of Hilbert polynomials of low dimensional Calabi Yau manifolds.  This can be an efficient way of computing the Hilbert polynomials:  In these examples, the Hilbert polynomials were computed by counting lattice points by inspection for small $k,$ and deducing the coefficients of the Hilbert polynomial from that data.  See \cite{qcy} for details.

\begin{corollary} The Hilbert polynomials of Calabi Yau hypersurfaces in low dimensional projective spaces are given by the following formulas:
\begin{itemize}

\item Hilbert polynomial of a torus in ${\mathbb CP}^2$ (cf. Section 3.3.1 of \cite{qcy}):

Let $Z_1 \subset {\mathbb CP}^2$ denote a smooth torus given by a homogeneous cubic polynomial.  Let $L \to {\mathbb CP}^2$ be the tautological line bundle.

Then the Hilbert polynomial of $Z_1$ is given by

$$ \ind{\bar{\partial}}_{(L^k|_{Z_1})} = 3k.$$
\item Hilbert polynomial of a K3 surface in ${\mathbb CP}^3$ (cf. Section 3.3.2 of \cite{qcy}):

Let $Z_2 \subset {\mathbb CP}^3$ denote a smooth K3 surface  given by a homogeneous quartic polynomial.  Let $L \to {\mathbb CP}^3$ be the tautological line bundle.

Then the Hilbert polynomial of $Z_2$ is given by

$$ \ind{\bar{\partial}}_{(L^k|_{Z_2}) }= 2k^2 + 2.$$
\item Hilbert polynomial of a quintic 3-fold in ${\mathbb CP}^4$ (cf. Section 3.3.3 of \cite{qcy}):

Let $Z_3 \subset {\mathbb CP}^4$ denote a smooth quintic 3-fold given by a homogeneous quintic polynomial.  Let $L \to {\mathbb CP}^4$ be the tautological line bundle.

Then the Hilbert polynomial of $Z_3$ is given by

$$ \ind{\bar{\partial}}_{(L^k|_{Z_3}) }= \frac56 k^3 + \frac{25}{6}k.$$\end{itemize}
\end{corollary}
\subsection{The Euler class in $K$ theory and its uses}

We now describe the main idea of the proof of Theorem \ref{mainthm}. 

Let $\pi_{M}: M \to \text{pt}$, $\pi_Z: Z \to \text{pt}$ be the constant maps. Since both $M$ and $Z$ are compact, complex manifolds, they each have a $K$ orientation, that is, a complex structure on the tangent bundle.  Considering the line bundles $L^k \to M$ ($k$ a positive integer) and $L^k|_Z \to Z$ as elements of $K(M)$ and $K(Z),$ respectively, we have
\[
\text{ind } \bar{\partial}_{L^k} = (\pi_{M})_! L^k
\]
and
\[
\text{ind } \bar{\partial}_{({L^k}|_Z)} = (\pi_Z)_! (L^k|_Z)
\]

But we may also write
\[
(\pi_Z)_! (L^k|_Z) = (\pi_{M})_! (L^k \otimes e^{ {K}}(\otimes_{i=1}^d L_{F_i}))
\]
where $e^{ {K}}$ is the Euler class in $K$ theory. For a line bundle
\[
e^{ {K}}(L) = (1-L^*)
\]

Thus
\[
\text{ind } \bar{\partial}_{({L^k}|_Z)} = (\pi_{M})_!\left(L^k \otimes (1-\bigotimes_{i=1}^d L_{F_i}^*)\right) 
\]

We now use the key computation of the paper, Proposition \ref{eclass}, to write
\[
1-\bigotimes_{i=1}^d L_i^* = \sum_{\ell=1}^d (-1)^{\ell+1} \sum_{\substack{I \subset \{1,\ldots,d\} \\ |I|=\ell}} \bigotimes_{i \in I} (1-L_{F_i}^*) 
\]

Thus
\begin{align*}
(\pi_Z)_!(L^k|_Z )&= \sum_{\ell=1}^d (-1)^{\ell+1} \sum_{\substack{I \subset \{1,\ldots,d\} \\ |I|=\ell}}  (\pi_{M})_! (L^k \otimes \bigotimes_{i \in I}  (1-L_{F_i}^*))\\
&= \sum_{\ell=1}^d (-1)^{\ell+1} \sum_{\substack{I \subset \{1,\ldots,d\} \\ |I|=\ell}}  (\pi_{D_I})_!( L^k|_{D_I} )
\end{align*}

\noindent where, for $I \subset \{1,\ldots,d\}$, $D_I = \bigcap_{i \in I} D_i$ and $\pi_{D_I}: D_I \to \text{pt}$ is the constant map on the subvariety $D_I \subset M$.

However, in Proposition \ref{smoothness}, we show
that or any $I \subset \{1,\ldots,d\}$, the subvariety $D_I = \bigcap_{i \in I} D_i$ is a smooth toric variety.
Hence, by Theorem \ref{tvp}, for any $I \subset \{1,\dots,d\},$
\[
(\pi_{D_I})_!( L^k|_{D_I} )= \#(k F_I \cap \mathbb{Z}^m)
\]
where $F_I = \bigcap_{i \in I} F_i$.

Therefore 
\[
(\pi_Z)_!(L^k|_Z) = \sum_{\ell=1}^d (-1)^{\ell+1} \sum_{\substack{I \subset \{1,\ldots,d\} \\ |I|=\ell}}  \#(k F_I \cap \mathbb{Z}^m)
\]

An application of the inclusion-exclusion principle gives 

\begin{equation}\labell{incexc}\sum_{\ell=1}^d (-1)^{\ell+1} \sum_{\substack{I \subset \{1,\ldots,d\} \\ |I|=\ell}}  \#(k F_I \cap \mathbb{Z}^m) =  \#(k (\cup_{i=1}^d F_i )\cap \Z^m) = \#(k\partial\Delta \cap \Z^m)\end{equation}

as needed.

Thus, remarkably, the formula of Proposition \ref{eclass} for the Euler class in $K$-theory of a tensor product of line bundles parallels the inclusion-exclusion formula (\ref{incexc}).  This is the main idea of this paper.

Theorem \ref{bdp} follows by a characteristic class computation, combined with the type of methods we recall in Section \ref{section2} to prove Khovanskii's formula; see Section \ref{section5}.

\subsection{Structure of the paper}

This paper is structured as follows. In Section \ref{section2} we recall the basic facts about toric varieties and outline the proofs of Theorem \ref{tvp} and Theorem \ref{hlp}. In Section \ref{section3}, we study the Euler class in $K$ theory, and give a formula (Proposition \ref{eclass}) for the Euler class in $K$-theory of a product of line bundles. In Section \ref{section4}, we apply this formula to prove Theorem \ref{mainthm}. We then use this result in Section \ref{section5} to give a geometric proof of Theorem \ref{bdp} (Theorem 1 of \cite{qcy}).

 \section{Lattice points in polytopes and Hilbert polynomials of toric varieties} \labell{section2}
 In this Section we recall the main results we will need about toric varieties, and sketch the proof of  Khovanskii's formula (Theorem \ref{hlp}). Variants of the methods we use here will arise in the proofs of Theorem \ref{mainthm} and Theorem \ref{bdp}.

Recall that $(M,\omega)$ is a smooth, compact, connected Kahler toric variety of real dimension $2m$  equipped with a a holomorphic Hermitian line bundle $L$ with Chern connection $\nabla$ of curvature $\omega.$ The manifold $M$ is equipped with an effective Hamiltonian action of a compact torus $T^m$, with moment map
\[
\mu: M \to \mathfrak{t}^* \cong \mathbb{R}^m
\]
of image given by a lattice polytope $\Delta \subset \mathbb{R}^m$. 

Recall also that we described the polytope $\Delta$  as an intersection of half spaces.  For $\lambda_i  \in \mathbb{R}_{\geq 0},$ $i=1,\dots,d,$ let
\[
  \Delta(\lambda_1 , \ldots, \lambda_d ) = \bigcap_{i=1}^d H_i(\lambda_i)
\]
be the intersection of half-spaces $H_i(\lambda_i)$ given by
\[
H_i(\lambda_i) = \{x \in \mathbb{R}^m : x \cdot n_i \leq \lambda_i\}, \quad i=1,\ldots,d
\]
where $n_i \in \mathbb{R}^m$ are primitive lattice vectors normal to the facets $F_i$ of $\Delta.$ Then for some nonnegative real numbers $\lambda_i^0,$ $ i = 1, \dots, d,$ 
we have

\[
\Delta = \Delta(\lambda_1^0 , \dots, \lambda_d^0 ).\] 

 In fact, we obtain a family of symplectic forms $\omega_{\lambda_1,\dots,\lambda_d}$ on $M$ for ${\lambda_1,\dots,\lambda_d}$ sufficiently close to $\lambda_1^0,\dots,\lambda_d^0,$ and corresponding moment maps whose images are the polytopes $\Delta({\lambda_1,\dots,\lambda_d}).$    

Now  consider the Dolbeault operator $\bar{\partial}_{L}$ and its index 
${\rm ind}(\bar{\partial}_{L}).$
 
Let

$$Td(x) = \frac{x}{1-e^{-x}}.$$

\noindent The function $Td(x)$ is analytic near the origin and has the power series expansion

\begin{equation}\labell{tde}Td(x) = 1 + \sum_{j=1}^\infty \frac{b_j}{j!}  x^j\end{equation}

\noindent where the coefficients $b_j$ are (up to signs) the Bernoulli numbers (See e.g. \cite{bourbaki}).

We recall the following facts from \cite{Fulton, Guillemin}.

\begin{RefTheorem}\labell{factstv} Let $(M,\omega)$ be a smooth compact connected Kahler toric variety, equipped with a holomorphic Hermitian line bundle $L$ with Chern connection of curvature $\omega,$ and with a Hamiltonian $T^m$ action with moment map $\mu: {M} \to \R^m$ with image $\mu(M) = \Delta.$  Let $ \Delta(\lambda_1 , \ldots, \lambda_d )$ and  $\omega_{\lambda_1,\dots,\lambda_d}$ be the deformations of $\Delta$ and $\omega$ as above. Then we have
\begin{enumerate}
\item The Riemann-Roch Theorem

$$ {\rm ind~} \bar{\partial}_{L } = \int_{M} Td(TM) e^{[\omega]}.$$

\item "Quantization Commutes with Reduction"

$$ {\rm ind~} \bar{\partial}_{L }  = \# (\Delta \cap \Z^n).$$

\item The Duistermaat-Heckman theorem

$${\rm vol} (\Delta(\lambda_1,\dots,\lambda_d)) = \int_{M} e^{\omega_{\lambda_1,\dots,\lambda_d}}$$

\noindent and 

$$[\omega_{\lambda_1,\dots,\lambda_d}] =\sum_{i=1}^d  \lambda_i c_1(L_{F_i})$$

\noindent where $L_{F_i}$ is the line bundle corresponding to the divisor given by $D_i = \mu^{-1}(F_i),$ and where $F_i$ is the $i$-th facet of $\Delta.$
\item The stable equivalence 

$$TM  \simeq \oplus_{i=1}^d L_{F_i}.$$

\end{enumerate}
\end{RefTheorem}
We now sketch the proof of Khovanskii's formula:  Combining (3) and (4),we have

$$\prod_{i=1}^d (Td(\frac{\partial}{\partial \lambda_i})) \int_{M} e^{\omega_{\lambda_1,\dots,\lambda_d}} = \int_{M}\prod_{i=1}^d (Td(c_1(L_{F_i})))e^{\omega_{\lambda_1,\dots,\lambda_d}} = \int_{M} Td(TM) e^{\omega_{\lambda_1,\dots,\lambda_d}}$$

\noindent and thus, using (1), (2), and (3), we obtain 

\begin{equation}\labell{kpff}\# (\Delta \cap \Z^n)= \prod_{i=1}^d (Td(\frac{\partial}{\partial \lambda_i}))  {\rm vol~}(\Delta(\lambda_1,\dots,\lambda_d))\vert_{\lambda_i =  \lambda_i^0} \end{equation}

\noindent where the infinite order differential operator $\prod_{i=1}^d (Td(\frac{\partial}{\partial \lambda_i}))$ is defined using the power series expansion (\ref{tde}) at the origin of the function $Td(x),$ and is applied to the polynomial $ {\rm vol~}(\Delta(\lambda_1,\dots,\lambda_d)).$
 
\section{The Euler class in $K-$theory}\labell{section3}

We begin with a computation of the Euler class in $K$ theory of a tensor product of line bundles.   Although this computation is very simple, it is the key element in our computation of Hilbert polynomials; it will allow us to express the Hilbert polynomial of the Calabi Yau hypersurface $Z$ in terms of the Hilbert polynomials of the smooth toric varieties corresponding to the faces of $\Delta.$  The decomposition we get parallels precisely the inclusion-exclusion principle for the number of lattice points in a union of facets of $\Delta.$  Since this computation is the crux of the proof, and will have further applications, we devote a separate section of the paper to it.

Let $X$ be a compact Hausdorff space. We denote by $K(X)$ the complex $K-$theory of $X.$  In this paper we will only consider the even $K-$group $K^0(X).$

Let $V \to X$ be a complex vector bundle. Then the Euler class $e^{K}(V)$ in $K$ theory is the $K$-class $e^{ {K}}(V)   \in  {K}(X)$ given by

\[
e^{ {K}}(V) = \Lambda^{*}(-V)^{*}
 = \sum_{k} (-1)^{k} \Lambda^{k} V^{*}
\]

In particular, if $L \to X$ is a line bundle

\[
e^{ {K}}(L) = 1 - L^{*}
\]
The following Proposition is the key to our main result.
\begin{Proposition}\labell{eclass}
Let $L = \bigotimes_{i=1}^{n} L_i$, where $L_i \to X$ are line bundles.

Then

\[
e^{ {K}}\left(L\right) = \sum_{\ell=1}^{n} (-1)^{\ell+1} \sum_{\substack{I \subset \{1,\ldots,d\} \\ |I|=\ell}} \bigotimes_{i \in I} (1 - L_i^{*})
= \sum_{\ell=1}^{n} (-1)^{\ell+1} \sum_{\substack{I \subset \{1,\ldots,d\} \\ |I|=\ell}} \prod_{i \in I} e^{K}(L_i)
\]
\end{Proposition}
\begin{proof}
We must show

\[
1 - \bigotimes_{i=1}^{n}  L_i^{*} = \sum_{\ell=1}^{n} (-1)^{\ell+1} \sum_{\substack{I \subset \{1,\ldots,d\} \\ |I|=\ell}} \bigotimes_{i\in I} (1 - L_i^{*})
\]

But

\[
\bigotimes_{i=1}^{n}  L_i^{*} = \bigotimes_{i=1}^{n} (1 - (1 - L_i^{*}))
 = \sum_{\ell=0}^{n} (-1)^{\ell} \sum_{\substack{I \subset \{1,\ldots,d\} \\ |I|=\ell}}  \bigotimes_{i \in I} (1 - L_i^{*}),
\]

\noindent just as for polynomials in $\mathbb{C}[x_1, \ldots, x_n]$

\[
\prod_{i=1}^{n}  x_i = \prod_{i=1}^{n} (1 - (1 - x_i))
= \sum_{\ell=0}^{n} (-1)^{\ell} \sum_{\substack{I \subset \{1,\ldots,d\} \\ |I|=\ell}}  \prod_{ i \in I} (1 - x_i).
\]

So

\[
1 - \bigotimes_{i=1}^{n}  L_i^{*} = \sum_{\ell=1}^{n} (-1)^{\ell+1} \sum_{\substack{I \subset \{1,\ldots,d\} \\ |I|=\ell}}  \bigotimes_{i \in I} (1 - L_i^{*}).
\]

More explicitly

\[
 1 - \bigotimes_{i=1}^{n}  L_i^{*}  = \sum_{1 \leq i \leq n} (1 - L_i^{*})
~-\sum_{1 \leq i < j \leq n} (1 - L_i^{*}) \otimes (1 - L_j^{*})
+ \cdots + (-1)^{n+1} (1 - L_1^{*}) \otimes \dots \otimes(1 - L_n^{*}).
\]

or 

\[
 e^{K}(\bigotimes_{i=1}^{n}  L_i) = \sum_{1 \leq i \leq n} e^{K}(L_i)
~- \sum_{1 \leq i < j \leq n} e^{K}(L_i)e^{K}(L_j)
+ \cdots + (-1)^{n+1} e^{K}(L_1)  \dots  e^{K}(L_n).
\]

\end{proof}

\section{Hilbert polynomials of Calabi Yau hypersurfaces and lattice points in polytope boundaries:  Proof of Theorem \ref{mainthm}}\labell{section4}

In this section we prove the main result of the paper.

Recall that $ (M ^{2m},\omega)$ is a smooth, compact, connected Kahler toric variety, equipped with a a holomorphic Hermitian line bundle $L$ with Chern connection $\nabla$ of curvature $\omega.$  Then a torus $T^m$ acts on $M$ in a Hamiltonian fashion, with moment map $\mu: {M} \to \mathbb{R}^{m}$ of image $\Delta$.

Recall also that $F_1, \ldots, F_d$ denote the facets of $\Delta$  and  ${L}_{F_i}$ are the line bundles corresponding to the divisors $D_i = \mu^{-1}(F_i)$.
 
We wish to prove

{\bf Theorem 1}.  {\em Suppose the divisor class $[D_{F_1}] + \dots + [D_{F_d}]$ corresponding to the line bundle $L_{F_1} \otimes \cdots \otimes {L_{F_d}}$  has a representative given by a smooth connected complex hypersurface $Z \subset M.$

Then the  Hilbert polynomial of $Z$ is given by
\[
\operatorname{ind} \overline{\partial}_{(L^k|_Z)} = \#(k (\partial \Delta) \cap \Z^m).
\]}
 
We first need the following result:

\begin{Lemma}\labell{smoothness}
Let $I \subset \{1, \ldots, d\}$ and let $F_I = \bigcap_{i \in I} F_i$ be a face of $\Delta$.   The divisor $D_I = \cap_{i \in I} D_i$ is a smooth complex submanifold of $M.$
\end{Lemma}

This lemma will follow from the following lemma by induction on codimension, since the divisor $D_i$ corresponding to a facet of $\Delta$ is itself a toric variety.

\begin{Lemma}
Let $F_i$ be a facet of $\Delta$. Then $D_i$ is a smooth complex submanifold of $M.$
\end{Lemma}

\begin{proof}
For each facet $F_i$ of $\Delta,$ there exists a codimension-one subtorus $S \subset T^m$ of $T^m,$ so that $D_i$ is a component of the fixed set of $S.$  Since the action of $T^m$ is holomorphic, so is the action of $S.$  So $D_i$ is a smooth complex submanifold of $M.$
\end{proof}

We now compute $\operatorname{ind} \overline{\partial}_{(L^k|_Z)}$.

Let $\pi_{M}: M \to \text{pt}$, $\pi_Z: Z \to \text{pt}$ be the constant maps. Since both $M$ and $Z$ are compact, complex manifolds, they each have a $K$ orientation, and the $K$ orientation on $Z$ is the one inherited from $M$:  Then the normal bundle $NZ$ to the subvariety $Z$ is given by $NZ = (L_{F_1} \otimes \dots \otimes L_{F_d})|_Z.$  Considering the line bundles $L^k \to M$ ($k$ a positive integer) and $L^k|_Z \to Z$ as elements of $K(M)$ and $K(Z),$ respectively, we have
\[
\text{ind } \bar{\partial}_{L^k} = (\pi_{M})_! L^k
\]
and
\[
\text{ind } \bar{\partial}_{({L^k}|_Z)} = (\pi_Z)_! (L^k|_Z)
\]

We then have, using Proposition \ref{eclass},
\begin{align*}
\operatorname{ind} \overline{\partial}_{(L^k|_Z)} &=  (\pi_Z)_! (L^k|_Z) = (\pi_{M})_! (L^k \otimes e^{K}(NZ)) \\
&= (\pi_{M})_!  \left( L^k \otimes \left( 1 - \bigotimes_{i=1}^d L_{F_i}^* \right) \right) \\
&= (\pi_{M})_!  \left( L^k \otimes \left(\sum_{\ell=1}^d (-1)^{\ell+1} \sum_{\substack{I \subset \{1,\ldots,d\} \\ |I|=\ell}} \bigotimes_{i \in I}(1 - L_{F_i}^*) \right)\right) \\
&= \sum_{\ell=1}^d (-1)^{\ell+1} \sum_{\substack{I \subset \{1,\ldots,d\} \\ |I|=\ell}} {\pi_{D_I}}_!( L^k|_{D_I})
\end{align*}

\noindent where $D_I = \cap_{i\in I} D_i$ and $\pi_{D_I} \to {\rm pt}$ is the constant map.
But $D_I$ is a smooth toric variety. Hence
\[
{\pi_{D_I}}_!(L^k|_{D_I}) = \#(k {F_I} \cap \mathbb{Z}^m)
\]
by Theorem \ref{tvp}.

Thus
\[
\operatorname{ind} \overline{\partial}_{(L^k|_Z)} = \sum_{\ell=1}^d (-1)^{\ell+1} \sum_{\substack{I \subset \{1,\dots,d\} \\ |I|=\ell}} \#(k F_I \cap \mathbb{Z}^m).
\]

But the inclusion-exclusion principle gives
\[
\sum_{\ell=1}^d (-1)^{\ell+1} \sum_{\substack{I \subset \{1,\dots,d\}\\ |I|=\ell}} \#\left(k \left(\bigcap_{i \in I} F_i\right) \cap \mathbb{Z}^m\right) = \#(k (\cup_{i=1}^d F_i) \cap \mathbb{Z}^m) = \#(k\partial  \Delta \cap \mathbb{Z}^m)
\]

\noindent as needed.

The core of the proof is the way in which the formula of Proposition \ref{eclass} for the Euler class in $K$-theory parallels the inclusion-exclusion principle in Combinatorics. We expect this parallel to have further applications.

\section{Formulas for lattice points in polytope boundaries and the Proof of Theorem \ref{bdp}}\labell{section5}

In \cite{qcy} we proved, by combinatorial method, a formula for the number of lattice points in the boundary of a Delzant polytope, and gave a moral argument for how such a formula should follow from geometric considerations. In this section we show how Theorem \ref{mainthm} gives rise to a geometric proof of these formulas.

\bigskip

Recall the power series expansions near zero of the functions 
\[
\hat{A}(x) = \frac{x/2}{\sinh(x/2)} =  \sum_{j=0}^\infty {c_{2j}}x^{2j}\]
 and  
\[
\quad \frac{1}{\hat{A}(x)} = \frac{\sinh(x/2)}{x/2} =  \sum_{j=0}^\infty \frac{1}{2^{2j}(2j+1)!} x^{2j}.
\]
 
We now prove Theorem \ref{bdp}.

\begin{proof}
Let $M$ be the Kahler toric variety associated with the polytope $\Delta$, and $L, L_{F_i}$ be as in Section \ref{section2}. Then by Theorem \ref{mainthm},
\[
\#(\partial\Delta \cap \mathbb{Z}^m) = \operatorname{ind} \bar{\partial}_{(L|_Z)}
= (\pi_{M})_!\left( L \otimes \left(1 - \bigotimes_{i=1}^{d}  L_{F_i}^*\right)\right)
\]

Since ${M}$ is smooth,
\begin{align*}
(\pi_{M})_!\left(L \otimes \left(1 - \bigotimes_{i=1}^{d} L_{F_i}^*\right) \right)
&= \int_{{M}} \operatorname{Td}(T{M}) \, e^{c_1(L)} \operatorname{ch}\left(1 - \bigotimes_{i=1}^{d} L_{F_i}^*\right) \\
&= \int_{M} \operatorname{Td}(T{M}) \, e^{c_1(L)} \left(1 - e^{-\sum_{i=1}^{d} c_1(L_{F_i})}\right)
\end{align*}

But $T{M} \cong \bigoplus_{i=1}^{d} L_{F_i}$ (by (4) of Theorem \ref{factstv}); so

\[
\operatorname{Td}(T{M}) = \prod_{i=1}^{d} \frac{c_i(L_{F_i})}{1-e^{-c_i(L_{F_i})}}
\]

Hence
\begin{align*}
&(\pi_{M})_! \left(L \otimes \left(1 - \bigotimes_{i=1}^{d} L_{F_i}^*\right)\right) \\
&= \int_{M} \left(\prod_{i=1}^{d} \frac{c_1(L_{F_i})}{1-e^{-c_i(L_{F_i})}} \right)e^{c_i(L)} \left(1 - e^{-\sum_{i=1}^{d} c_i(L_{F_i})}\right) \\
&= \int_{M} \left(\prod_{i=1}^{d} \frac{c_i(L_{F_i})}{e^{c_1(L_{F_i})/2} - e^{-c_1(L_{F_i})/2}} \right)\frac{1}{\prod_{i=1}^{d} e^{-c_1(L_{F_i})/2}}  \left(1 - e^{-\sum_{i=1}^{d} c_1(L_{F_i})}\right) e^{c_1(L)}\\
&= \int_{M} \prod_{i=1}^{d} \hat{A}(c_1(L_{F_i})) \left( \frac{1}{\hat{A}}\left(\sum_{i=1}^{d} c_1(L_{F_i})\right) \right)\left(\sum_{i=1}^{d} c_1(L_{F_i})\right) e^{c_i(L)}
\end{align*}

Recall also that for $\lambda_i$ near $\lambda_i^0,$ $i = 1,\dots,d,$ we may equip $M$ with a symplectic form $\omega_{\lambda_1,\dots,\lambda_d}$ of cohomology class
 
\[
[\omega_{\lambda_1,\dots,\lambda_d}]= \sum_{i=1}^{d} \lambda_i c_1(L_{F_i});
\]

\noindent giving rise to the moment image $\Delta(\lambda_1,\dots,\lambda_d).$

By the Duistermaat-Heckman Theorem (Item (3) in Theorem \ref{factstv})
\[
\int_{M}e^{[\omega_{\lambda_1,\dots,\lambda_d}]} = \operatorname{vol}(\Delta(\lambda_1,\ldots,\lambda_d))
\]

and 

$$\frac{\partial}{\partial \lambda_i} [\omega_{\lambda_1,\dots,\lambda_d}] = c_1(L_{F_i}).$$

Hence
\begin{align*}
&\int_{M} \prod_{i=1}^{d}  {\hat{A}}(c_1(L_{F_i})) \left( \frac{1}{\hat{A}}\left(\sum_{i=1}^{d} c_1(L_{F_i}) \right) \right)\left(\sum_{i=1}^{d} c_1(L_{F_i})\right) e^{c_1(L)} \\
&= \prod_{i=1}^{d} \hat{A}\left(\frac{\partial}{\partial \lambda_i}\right) \left(\frac{1}{\hat{A}}\left(\sum_{i=1}^{d} \frac{\partial}{\partial \lambda_i}\right)\right)\left(\sum_{i=1}^{d} \frac{\partial}{\partial \lambda_i}\right) \operatorname{vol}(\Delta(\lambda_1,\ldots,\lambda_d))\bigg|_{\lambda_i = \lambda_i^0}
\end{align*}

But
\[
\sum_{i=1}^{d} \frac{\partial}{\partial \lambda_i} \operatorname{vol}(\Delta(\lambda_1,\ldots,\lambda_d)) = \operatorname{vol}(\partial\Delta(\lambda_1,\ldots,\lambda_d))
\]
\noindent where $\text{vol}(\partial\Delta(\lambda_1,\dots,\lambda_d))$ is defined to be the sum of the Euclidean volumes of the facets, each divided by the length of the primitive lattice vector normal to it (See e.g. \cite{robins}, Lemma 5.18), proving the Theorem.
\end{proof}

\end{document}